\documentclass[12pt, a4paper, twoside, reqno]{amsart}

\usepackage[margin=2.85cm]{geometry}

\usepackage[T1]{fontenc}
\usepackage[latin1]{inputenc}
\usepackage{amsmath}
\usepackage{amsthm}
\usepackage{amsfonts}
\usepackage{MnSymbol}
\usepackage{bbm}
\usepackage[UKenglish]{babel}
\usepackage{enumerate}
\usepackage[dvips]{graphicx}
\usepackage{bm}
\usepackage{ae}

\theoremstyle{definition}

\theoremstyle{remark}

\theoremstyle{plain}
\newtheorem{thm}{Theorem}
\newtheorem{lem}[thm]{Lemma}

\newtheorem{cor}[thm]{Corollary}

\newcommand{\norm}[1]{\ensuremath{\left\Vert #1 \right\Vert}}
\newcommand{\abs}[1]{\ensuremath{\left\vert #1 \right\vert}}

\DeclareMathOperator{\Bad}{Bad}
 
 \newcommand{\R}{\mathbb{R}}
 \newcommand{\N}{\mathbb{N}}
 \newcommand{\Q}{\mathbb{Q}}
 
 \newcommand{\rar}{\rightarrow}

\begin{document}

\title{Metrical musings on Littlewood and friends}

\author{A. HAYNES}

\address{A. Haynes, Department of Mathematics, University of Bristol,
  University Walk, Bristol, BS8 1TW, England}

\email{alan.haynes@bris.ac.uk}

\author{J. L. JENSEN}

\address{J. L. Jensen, Department of Mathematics, Aarhus University,
  Ny Munkegade 118, DK-8000 Aarhus C, Denmark}

\email{jonas@imf.au.dk}

\author{S. KRISTENSEN}

\address{S. Kristensen, Department of Mathematics, Aarhus University,
  Ny Munkegade 118, DK-8000 Aarhus C, Denmark}

\email{sik@imf.au.dk}

\thanks{AH supported by EPSRC grant EP/J00149X/1. SK's research supported by the Danish Research Council for
  Independent Research.}


\begin{abstract}
  We prove a metrical result on a family of conjectures related to the
  Littlewood conjecture, namely the original Littlewood conjecture,
  the mixed Littlewood conjecture of de Mathan and Teuli\'e and a
  hybrid between a conjecture of Cassels and the Littlewood
  conjecture. It is shown that the set of numbers satisfying a strong
  version of all of these conjectures is large in the sense of
  Hausdorff dimension restricted to the set of badly approximable
  numbers.
\end{abstract}

\maketitle

\section{Introduction}
\label{sec:introduction}

The Littlewood conjecture in Diophantine approximation
is concerned with the simultaneous approximation of two real numbers
by rationals with the same denominator. It states that for any pair
$\alpha, \beta \in \mathbb{R}$,
\begin{equation}
  \label{eq:1}
  \liminf_{q\in\N} q \norm{q \alpha} \norm{q \beta} = 0,
\end{equation}
where $\|\cdot\|$ denotes the distance to the nearest integer. It
follows from Dirichlet's theorem that for any single real number
$\alpha$,
\begin{equation}
  \label{eq:2}
  \liminf_{q\in\N} q \norm{q \alpha} \leq 1.
\end{equation}
This is best possible, apart from improvements in the constant on the
right hand side. Indeed if $\alpha\not\in\Q$ has bounded partial
quotients in its simple continued fraction expansion, then the left
hand side of \eqref{eq:2} is positive. Such a number $\alpha$ is
called \emph{badly approximable}, and we denote the collection of
badly approximable numbers by $\Bad\subseteq\R$.

By the Borel-Bernstein theorem (i.e. almost all real numbers have
unbounded partial quotients) the Lebesgue measure of $\Bad$ is $0$,
but by a theorem of Jarnik \cite{jarnik28:_zur_theor_approx} its
Hausdorff dimension is $1$. Furthermore since \eqref{eq:1} is
satisfied whenever $\alpha$ or $\beta$ lies in $\R\setminus\Bad$, it
follows immediately that the set of exceptions to the Littlewood
conjecture has Lebesgue measure $0$.

The Littlewood conjecture has a long history. One of the first
results, by Cassels and Swinner\-ton-Dyer \cite{MR0070653}, states
that the conjecture is satisfied if $\alpha$ and $\beta$ lie in the
same cubic extension of $\mathbb{Q}$. Although there have been several
further advances toward this conjecture, the most widely quoted recent
result is that of Einsiedler, Katok, and Lindenstrauss
\cite{MR2247967}, who used powerful measure rigidity theorems in the
space of unimodular lattices to show that the set of $(\alpha,
\beta)\in\R^2$ for which \eqref{eq:1} fails has Hausdorff dimension
$0$.

However one should not forget that this breakthrough was preceded by
an equally significant theorem of Pollington and Velani
\cite{MR1819996}, which states that for any $\alpha \in \Bad$, there
is a set $G \subseteq \Bad$ of Hausdorff dimension $1$ such that, for
all $\beta \in G$,
\[\liminf_{q\rar\infty}q(\log q)\|q\alpha\|\|q\beta\|\le 1.\]
The conclusion here is stronger than \eqref{eq:1} and it is worth
pointing out that, although it does not say as much about the
exceptional set in the Littlewood conjecture, this result does not
follow from the theorem of Einsiedler, Katok, and Lindenstrauss. More
explicit constructions in the spirit of Pollington and Velani's
approach have been obtained by Adamczewski and Bugeaud
\cite{MR2225491}.

A problem related to the Littlewood conjecture is the so-called mixed
Littlewood conjecture of de Mathan and Teuli\'e \cite{MR2103807}. In
its most basic form, the conjecture states that for any $\alpha \in
\mathbb{R}$ and any prime $p$,
\begin{equation}
  \label{eq:3}
  \liminf_{q \rightarrow \infty} q \abs{q}_p \norm{q \alpha} = 0,
\end{equation}
where $\abs{q}_p$ denotes the $p$-adic absolute value of $q$. As
above, \eqref{eq:3} is trivially satisfied unless
$\alpha \in \Bad$.

The mixed Littlewood conjecture has also attracted considerable
interest in recent years. Like the Littlewood conjecture, it has a
nice interpretation in terms of homogeneous dynamics. In additional
there are new techniques from $p$-adic analysis and one dimensional
dynamics which are available for studying this problem. The original
statement of the conjecture is actually slightly more general then
what we have mentioned so far. It involves the notion of a
\emph{pseudo-absolute value}, which we attend to presently.

A pseudo-absolute value is defined in terms of a sequence $\mathcal{D}
= (n_k)_{k=0}^\infty$, satisfying the conditions $n_0=1$ and $n_{k-1}
\vert n_k$ for all $k$. The associated pseudo-absolute value of an
integer $q$ is then defined as
\begin{equation}
  \label{eq:4}
  \abs{q}_{\mathcal{D}} = \min\{n_k^{-1} : q \in n_k\mathbb{Z}\}.
\end{equation}
For an integer $a$ if we set $n_k = a^k$, then associated
pseudo-absolute value becomes the usual $a$-adic absolute value. The
more general version of the mixed Littlewood conjecture is the
assertion that
\begin{equation}
  \label{eq:5}
  \liminf_{q\in\N} q \abs{q}_{\mathcal{D}} \norm{q \alpha} = 0,
\end{equation}
for any $\alpha \in \mathbb{R}$ and any pseudo-absolute value sequence
$\mathcal{D}$.

We remark that with an additional $p$-adic absolute value multiplied
onto the left hand side of \eqref{eq:5}, and with a mild growth
condition on $\mathcal{D}$, it is proved in \cite{harrap:_littl} that
\begin{equation*}
  \liminf_{q\in\N} q \abs{q}_p \abs{q}_{\mathcal{D}}
  \norm{q \alpha} = 0,
\end{equation*}
for all $\alpha \in \mathbb{R}$. This shows that
a slightly weaker statement than the mixed Littlewood conjecture is
certainly true.

A final problem related to the Littlewood conjecture is the an old
conjecture of Cassels, recently resolved by Shapira \cite{MR2753608}.
Shapira's theorem states that for Lebesgue-almost every $\alpha,
\beta$,
\begin{equation*}
  \liminf_{q \rightarrow \infty} q \norm{q\alpha - \gamma_1} \norm{q
    \beta - \gamma_2} = 0,
\end{equation*}
for all $\gamma_1, \gamma_2\in\R$. In our main result below we address
the case when $\gamma_1 = 0$ and consider more closely the collection
of $\alpha,\beta$ for which
\begin{equation}
  \label{eq:10}
  \liminf_{q \rightarrow \infty} q \norm{q\alpha} \norm{q \beta -
    \gamma} = 0,
\end{equation}
for every value of $\gamma$.  Of course as before the problem is
trivial unless $\alpha \in \Bad$. We will call \eqref{eq:10} the
hybrid Cassels--Littlewood equation.

Now we present our main result, in which we consider a simultaneous
version of the three above mentioned problems.

\begin{thm}
  \label{thm:main}
  Fix $\epsilon > 0$ and let $\{\alpha_i\} \subseteq \Bad$ be a
  countable set of badly approximable numbers, and $\{\mathcal{D}_j\}$
  a countable set of pseudo-absolute value sequences. Then there is
  set of $G\subseteq \Bad$ of Hausdorff dimension $1$ such that for
  any $\beta\in G$,\vspace*{.1in}
  \begin{itemize}
  \item[(i)]  For any $i \in \mathbb{N}$ and $\gamma
  \in \mathbb{R}$ the inequality
  \begin{equation}
    \label{eq:6}
    q \norm{q\alpha_i} \norm{q \beta - \gamma} < \frac{1}{(\log
      q)^{1/2-\epsilon}},
  \end{equation}
  has infinitely many solutions $q\in\N$, and\vspace*{.1in}

  \item[(ii)] For any $j \in \mathbb{N}$ and $\delta \in \mathbb{R}$ we have that
  \begin{equation}
    \label{eq:7}
    \liminf_{ q\rightarrow \infty} q \abs{q}_{\mathcal{D}_j} \norm{q
      \beta - \delta} = 0.
  \end{equation}
    \end{itemize}\vspace*{.1in}
    Furthermore, for each $j$ such that $\mathcal{D}_j = (n_k)$
    satisfies the inequality $n_k \leq C^k$ for some $C>1$, we may
    replace \eqref{eq:7} by the stronger statement that for any
    $\delta\in\R$ the inequality
  \begin{equation}
    \label{eq:9}
    q \abs{q}_{\mathcal{D}_j} \norm{q \beta - \delta} < \frac{1}{(\log
      q)^{1/2-\epsilon}},
  \end{equation}
  has infinitely many solutions $q\in\N$.
\end{thm}

We should stress that while the Littlewood conjecture and the mixed
Littlewood conjecture are trivial whenever $\beta \notin \Bad$, this
is no longer the case for the inhomogeneous versions in the theorem.
Hence, the content of our result is of interest not only for badly
approximable numbers. Luckily, applying an argument analogous to our
proof below with the discrepancy estimate of R. C. Baker
\cite{MR623668} in place of our Corollary \ref{cor2}, we could get
Theorem \ref{thm:main} for Lebesgue-almost every $\beta$. For
comparison with Shapira's result \cite{MR2753608} this remains worth
stating.

A weaker version of our theorem is the following corollary.

\begin{cor}
  Let $\{\alpha_i\} \subseteq \Bad$ be a countable set of badly
  approximable numbers and $\{\mathcal{D}_j\}$ a countable set
  of pseudo-absolute value sequences. The set of $\beta \in \Bad$ for
  which the $\mathcal{D}_j$-mixed Littlewood conjectures are all
  satisfied and for which all pairs $(\alpha_i, \beta)$ satisfy the
  hybrid Cassels--Littlewood equation is of Hausdorff dimension $1$.
\end{cor}

Considering only the cases where $\gamma = 0$, we get the following.

\begin{cor}
  \label{cor:main}
  Let $\{\alpha_i\} \subseteq \Bad$ be a countable set of badly
  approximable numbers and $\{\mathcal{D}_j\}$ a countable set
  of pseudo-absolute value sequences. The set of $\beta \in \Bad$ for
  which the $\mathcal{D}_j$-mixed Littlewood conjectures are all
  satisfied and for which all pairs $(\alpha_i, \beta)$ satisfy the
  Littlewood conjecture is of Hausdorff dimension $1$.
\end{cor}

We point out that while these corollaries can be deduced using currently available techniques concerning measure rigidity in homogeneous spaces, our main theorem can not.

Our main result generalizes the result of Pollington and Velani
\cite{MR1819996} to cover both a countable number of $\alpha$'s and a
countable number of pseudo-absolute values. We also believe that, although we use some of the same techniques, our proof is substantially simpler.

We will deduce Theorem \ref{thm:main} from a result on the discrepancy
of certain sequences defined in terms of a generic point with respect
to a very special measure called Kaufman's measure. In the next section, we will
describe the properties of this measure and state the result on
uniform distribution. In Section \ref{sec:kaufm-meas-discr}, we will
prove Theorem \ref{thm:main}.  Finally, we give a few concluding
remarks.

Throughout we will use the Vinogradov notation and write $f \ll g$ for
two real quantities $f$ and $g$ if there exists a constant $C>0$ such
that $f \leq Cg$. If $f \ll g$ and $g \ll f$, we will write $f \asymp
g$. We will also as usual define the function $e(x) = \exp(2 \pi i x)$.

\section{Kaufman's measure and discrepancy}
\label{sec:kaufm-meas-discr}

The key tool in proving our main theorem is a result on the
discrepancy of certain sequences, which holds true for almost all
$\alpha$ with respect to a certain measure introduced by Kaufman
\cite{MR610711}.

Kaufman's measure $\mu_M$ is a measure supported on the set of real
numbers with partial quotients bounded above by $M$. To be explicit,
for each real number $\alpha \in [0,1)$, let
\begin{equation*}
  \alpha = [a_1, a_2, \dots] = \cfrac{1}{a_1 +\tfrac{1}{a_2 +
      \tfrac{1}{\dots}}}
\end{equation*}
be the simple continued fraction expansion of $\alpha$. For $M \geq
3$, let
\begin{equation}
  \label{eq:8}
  F_M = \{\alpha \in [0,1) : a_i(\alpha) \leq M \text{ for all }
    i \in \mathbb{N}\}.
\end{equation}

Kaufman \cite{MR610711} proved that the set $F_M$ supports a measure
$\mu_M$ satisfying a number of nice properties. For our purposes, we
need the following two properties.
\begin{enumerate}[(i)]
\item\label{item:1} For any $s < \dim(F_M)$, there are positive
  constants $c,l > 0$ such that for any interval $I \subseteq [0,1)$
  of length $\abs{I} \leq l$,
  \begin{equation*}
    \mu_M(I) \leq c \abs{I}^s.
  \end{equation*}
\item\label{item:2} For any $M$, there are positive constants $c, \eta
  >0$ such that the Fourier transform $\hat{\mu}_M$ of the Kaufman
  measure $\mu_M$ satisfies
  \begin{equation*}
    \hat{\mu}_M(u) \leq c \abs{u}^{-\eta}.
  \end{equation*}
\end{enumerate}
The first property allows us to connect the Kaufman measure with the
Hausdorff dimension of the set $F_M$ via the Mass Distribution
Principle. The second property provides a positive lower bound on the
Fourier dimension of the set $F_M$, but for our purposes the property
is used only in computations.

The second key tool is the notion of discrepancy from the theory of
uniform distribution. The discrepancy of a sequence in $[0,1)$
measures how uniformly distributed a sequence is in the interval.
Specifically, the discrepancy of the sequence $(x_n)$ is defined as
\begin{equation*}
  D_N(x_n) = \sup_{I \subseteq [0,1]} \abs{\sum_{n=1}^N \chi_I(x_n)
    - N \abs{I}},
\end{equation*}
where $I$ is an interval and $\chi_I$ is the corresponding
characteristic function. A sequence $(x_n)$ is uniformly distributed
if $D_N(x_n) = o(N)$.

Our key result is the following discrepancy estimate.

\begin{thm}
  \label{thm2}
  Let $\mu_M$ be a Kaufman measure and assume that for positive
  integers $u<v$ we have
  \begin{equation*}
    \sum_{n,m = u}^v \abs{a_n - a_m}^{-\eta} \ll \frac{1}{\log v}
    \sum_{n=u}^v \psi_n
  \end{equation*}
  where $(\psi_n)$ is a sequence of non-negative numbers and $\eta >
  0$ is the constant from property \eqref{item:2} of the Kaufman
  measure. Then for $\mu_M$-almost every $x \in [0,1]$ we have
  \begin{equation*}
    D_N(a_n x) \ll (N\log(N)^2 + \Psi_N)^{1/2} \log(N\log(N)^2 +
    \Psi_N)^{3/2 + \varepsilon} + \max_{n \leq N} \psi_n
  \end{equation*}
  where $\Psi_N = \psi_1 + \cdots + \psi_N$.
\end{thm}

\section{Proofs}
\label{sec:proof-main-theorem}

Initially, we prove Theorem \ref{thm2} and then proceed to deduce
Theorem \ref{thm:main}. For the proof of Theorem \ref{thm2}, we will
need the following result found in \cite{MR1672558}.

\begin{lem}
  \label{lem1}
  Let $(X,\mu)$ be a measure space with $\mu(X) < \infty$. Let
  $F(n,m,x),~ n,m\ge 0$ be $\mu$-measurable functions
  and let $\phi_n$ be a sequence of real numbers such that
  $\abs{F(n-1,n,x)} \leq \phi_n$ for $n\in\N$.  Let $\Phi_N =
  \phi_1 + \cdots + \phi_N$ and assume that $\Phi_N \to \infty$.
  Suppose that for $0 \leq u < v$ we have
  \begin{equation*}
    \int_X \abs{F(u,v,x)}^2 d\mu \ll \sum_{n=u}^v \phi_n.
  \end{equation*}
  Then for $\mu$-almost all $x$, we have
  \begin{equation*}
    F(0,N,x) \ll \Phi_N^{1/2} \log(\Phi_N)^{3/2+\varepsilon} +
    \max_{n \leq N} \phi_n.
  \end{equation*}
\end{lem}

We will also need the Erd\H{o}s--Tur\'an inequality (see e.g.
\cite{MR1297543}).

\begin{thm}
  \label{thm:E-T}
  For any positive integer $K$ and any sequence $(x_n) \subseteq
  [0,1)$,
  \begin{equation*}
    D_N(x_n) \leq \frac{N}{K+1} + 3 \sum_{k=1}^K \frac{1}{k}
    \abs{\sum_{n=1}^N e(k x_n)}.
  \end{equation*}
\end{thm}

\begin{proof}[Proof of Theorem \ref{thm2}]
  Suppose $M\ge 3$ and for integers $0 \leq u < v$ let
  \begin{equation*}
    F(u,v,x) = \sum_{h=1}^v \frac{1}{h} \abs{\sum_{n=u}^v e(h a_n
      x)}.
  \end{equation*}
  Theorem \ref{thm:E-T} with $K=N$ tells us that
  \[D_N(x_n)\ll F(0,N,x).\]
 Integrating with respect to $d\mu_M(x)$ and applying Cauchy-Schwartz gives
  \begin{align*}
    \int \abs{F(u,v,x)}^2 d\mu_M &\leq \sum_{h,k=1}^v
    \frac{1}{hk} \int \abs{\sum_{n=u}^v e(h a_n
      x)}^2 d\mu_M \\
    &= \sum_{h,k=1}^v \frac{1}{hk} \Bigg(v-u+1 + \mathop{\sum_{n,m =
        u}^v}_{n \neq m} \hat\mu_M(h(a_n - a_m)) \Bigg).
  \end{align*}
  Finally using property \eqref{item:2} of the Kaufman measure we have
  \begin{align*}
    \int \abs{F(u,v,x)}^2 d\mu_M &\ll \sum_{h,k=1}^v \frac{1}{hk}
    \Bigg(v-u+1 +  h^{-\eta} \mathop{\sum_{n,m = u}^v}_{n \neq m}
    \abs{a_n-a_m}^{-\eta} \Bigg) \\ &
    \ll \sum_{n=u}^v [\log(n)^2 + \psi_n].
  \end{align*}
  Since $F(n-1,n,x) \ll \log(n)^2 + \psi_n$ for all $n \geq
  1$, the theorem then follows from Lemma \ref{lem1}.
\end{proof}

We now state a corollary, which is of interest in its own right.
Assuming nothing about the sequence $(a_n)$ we get the following
corollary.
\begin{cor}
  \label{cor1}
  Let $\mu$ be a Kaufman measure. For $\mu$-almost every $x \in [0,1]$
  we have $D_N(a_n x) \ll N^{1-\nu}$ for some $\nu > 0$.  In
  particular $(a_n x)$ is uniformly distributed modulo 1.
\end{cor}

\begin{proof}
  Without loss of generality we may assume that $a_n < a_{n+1}$, and we have that
  \[
  \sum_{n,m = u}^v \abs{a_n-a_m}^{-\eta} \leq 2\sum_{m=u}^{v-1}
  \sum_{n=m+1}^{v} \abs{n-m}^{-\eta} \ll v^{2-\eta} - u^{2-\eta} \ll
  \frac{1}{\log v} \sum_{n=u}^v n^{1-2\nu'}
  \]
  for some $\nu' > 0$, so for $\mu$-a.e. $x$ we have
  \[ D_N(a_n x) \ll (N \log(N)^2 + N^{2-2\nu'})^{1/2}(\log(N \log(N)^2
  + N^{2-2\nu'}))^{3/2 + \varepsilon} + N^{1-2\nu'} \\ \ll N^{1-\nu}
  \] for any $\nu > 0$ with $\nu < \nu'$.
\end{proof}

Corollary \ref{cor1} is sufficient to prove Corollary \ref{cor:main}
directly. However, we are aiming for a proof of the more general
Theorem \ref{thm:main}. For the purposes of this result, we
can specialize to lacunary sequences $a_n$.

\begin{cor}
  \label{cor2}
  Let $\nu > 0$ and let $\mu$ be a Kaufman measure and $(a_n)$ a
  lacunary sequence of integers. For $\mu$-almost every $x \in [0,1]$
  we have $D_N(a_n x) \ll N^{1/2} (\log N)^{5/2 + \nu}$.
\end{cor}

\begin{proof}
  We apply again Theorem \ref{thm2}. Using lacunarity of the sequence
  $(a_n)$, we see that
  \begin{equation*}
    \sum_{n,m=1}^\infty \abs{a_n - a_m}^{-\eta} < \infty.
  \end{equation*}
  Consequently, we can absorb all occurrences of $\Psi_N$ as well as
  the final term $\max_{n \leq N} \psi_n$ in the discrepancy estimate
  of Theorem \ref{thm2} into the implied constant. It follows that
  \begin{equation*}
    D_N(a_n x) \ll (N\log(N)^2)^{1/2} \log(N\log(N)^2)^{3/2 +
      \varepsilon} \ll N^{1/2} (\log N)^{5/2 + \nu}
  \end{equation*}
  for $\mu$-almost every $x$, where $\nu$ can be made as small as
  desired by picking $\varepsilon$ small enough.
\end{proof}

We are now ready to prove the main result, Theorem \ref{thm:main}.
\begin{proof}[Proof of Theorem \ref{thm:main}]
  Let $G$ denote the set from the statement of the theorem
  and suppose, contrary to what we are to prove, that $\dim g < 1$. Pick an $M\ge 3$
  such that $\dim F_M > \dim G$ (this can be done in light of Jarn\'ik's Theorem). Let
  $\mu = \mu_M$ denote the Kaufman measure on $F_M$.

  Consider first one of the $\alpha_i$, and let $(q_k)$ denote the
  sequence of denominators of convergents in the simple continued fraction expansion of $\alpha_i$. The sequence $q_k$ is lacunary hence by Corollary \ref{cor2}, for $\mu-$almost every $x$,
  \begin{equation*}
    D_N(q_n x) \ll N^{1/2} (\log N)^{5/2 + \nu}.
  \end{equation*}

  Let $\psi(N) = N^{-1/2+\epsilon}$ for some $\epsilon > 0$ and
  consider the interval 
  \[I^\gamma_N = [\gamma - \psi(N), \gamma +
  \psi(N)].\]
  By the definition of discrepancy, for every $\gamma \in
  [0,1]$ and $\mu$-almost every $\beta$
  \begin{equation*}
    \abs{\#\{k \leq N : \{q_k \beta\} \in I^\gamma_N\} - 2N\psi(N)}
    \ll  N^{1/2} (\log N)^{5/2 + \nu}.
  \end{equation*}
  Hence,
  \begin{alignat*}{2}
    \#\{k \leq N : \{q_k \beta\} \in I^\gamma_N\} &\geq 2N\psi(N) - K
    N^{1/2} (\log N)^{5/2 + \nu}\\
    &= 2 N^{1/2+\epsilon} - KN^{1/2} (\log N)^{5/2 + \nu},
  \end{alignat*}
  where $K>0$ is the implied constant from Corollary \ref{cor2}. Next let $N^\gamma_h$ denote the increasing sequences defined by
  \begin{equation*}
    N^\gamma_h = \min\big\{N \in \mathbb{N} : \#\{k \leq N : \{q_k
    \beta\} \in I^\gamma_N\} = h \big\}.
  \end{equation*}

  We claim that the sequence $q_{N^\gamma_h}$ satisfies \eqref{eq:6} for the
  given $\alpha_i$ and every $\gamma$ with $\mu$-almost every $\beta$.
  Indeed, since each $q_{N^\gamma_h}$ is a denominator of a convergent
  to $\alpha_i$,
  \begin{equation*}
    q_{N^\gamma_h} \norm{q_{N^\gamma_h} \alpha_i} \leq 1.
  \end{equation*}
  Hence,
  \begin{equation*}
    q_{N^\gamma_h} \norm{q_{N^\gamma_h} \alpha_i} \norm{q_{N^\gamma_h}
      \beta - \gamma} \leq \norm{q_{N^\gamma_h} \beta - \gamma} \leq
    \left(N^\gamma_h\right)^{-1/2+\epsilon}.
  \end{equation*}
  Since $\alpha_i \in \Bad$, the sequence of denominators of
  convergents $q_n$ is bounded between the Fibonacci sequence and
  $(2M)^n$, where $M$ is the maximal partial quotient in the simple
  continued fraction expansion of $\alpha_i$. Hence, $n \asymp \log
  q_n$, and it follows that
  \begin{equation*}
    q_{N^\gamma_h} \norm{q_{N^\gamma_h} \alpha_i} \norm{q_{N^\gamma_h}
      \beta - \gamma} \leq (\log q_{N^\gamma_h})^{-1/2+\epsilon/2},
  \end{equation*}
  whenever $h$ is large enough. This establishes our claim and shows that the exceptional set $E_i
  \subseteq F_M$ for which \eqref{eq:6} does not hold has $\mu(E_i) =
  0$.

  Consider now one of the absolute value sequences, $\mathcal{D_j} =
  \{r_k\}$. If we do not have an upper geometric growth rate of the
  sequence, applying Corollary \ref{cor1} implies that \eqref{eq:7}
  holds for $\mu$-almost every $\beta$. Indeed, let $\beta$ be such
  that $\{r_k \beta\}$ is uniformly distributed and suppose to the
  contrary that for some $\delta$ and $\epsilon > 0$,
  \begin{equation*}
    r_k \abs{r_k}_{\mathcal{D}_j} \norm{r_k \beta - \delta} > \epsilon,
  \end{equation*}
  for every $k$. Since $r_k \abs{r_k}_{\mathcal{D}_j} = 1$, this would
  violate the uniform distribution of the sequence $\{r_k \beta\}$ and
  complete the proof of the statement.

  Next we show that when $r_k$ has the right upper growth
  rate, the stronger statement \eqref{eq:9} holds.  By the
  divisibility condition, this sequence must again be lacunary. Hence,
  \begin{equation*}
    D_N(r_n x) \ll N^{1/2} (\log N)^{5/2 + \nu}.
  \end{equation*}
  Defining $I^\delta_N$ with $\delta$ in place of $\gamma$ and
  $\Psi(N)$ as before and repeating the argument from the
  Littlewood-case, we see that
  \begin{alignat*}{2}
    \#\{k \leq N : \{q_k \beta\} \in I^\delta_N\} &\geq 2N\psi(N) - K
    N^{1/2}
    (\log N)^{5/2 + \nu}\\
    &= 2 N^{1/2+\epsilon} - KN^{1/2} (\log N)^{5/2 + \nu}.
  \end{alignat*}
As before, we take a sequence $N^\delta_h$
  and note that since $r_{N^\delta_h}
  \abs{r_{N^\delta_h}}_{\mathcal{D}_j} = 1$, we immediately have for
  $\mu$-almost every $\beta$ a sequence $r_{N^\delta_h}$ satisfying
  the inequalities \eqref{eq:9}. Indeed,
  \begin{equation*}
    r_{N^\delta_h} \abs{r_{N^\delta_h}}_{\mathcal{D}_j}
    \norm{r_{N^\delta_h} \beta - \delta} \leq \norm{r_{N^\delta_h}
      \beta - \delta} \leq \left(N^\delta_h\right)^{-1/2+\epsilon}.
  \end{equation*}
  Since the upper and lower growth assumptions imply that $k \asymp
  \log r_k$, this implies that
  \begin{equation*}
    r_{N^\delta_h} \abs{r_{N^\delta_h}}_{\mathcal{D}_j}
    \norm{r_{N^\delta_h} \beta - \delta} \leq \left(\log
      q_{N^\delta_h}\right)^{-1/2+\epsilon},
  \end{equation*}
  whenever $h$ is large enough. The upshot is that the exceptional set
  $E'_j \subseteq F_M$ for which \eqref{eq:7} does not hold has
  $\mu(E'_j) = 0$.

  To conclude, let $E$ be the set of $\beta \in F_M$ for which there
  is an $i$ or a $j$ such that either \eqref{eq:1} or \eqref{eq:5} is
  not satisfied. Then,
  \begin{equation*}
    E = \bigcup_i E_i \cup \bigcup_j E'_j,
  \end{equation*}
  and therefore $\mu(E) = 0$.

  Finally $\mu(G)$ is maximal, so consider the trace measure
  $\tilde{\mu}$ of $\mu$ on $G$, defined by
  $\tilde{\mu}(X) = \mu(X \cap G)$. It follows from property
  \eqref{item:1} of Kaufman's measures that $\mu$ is a mass distribution on $[0,1)$, and
  since $G$ is full, $\tilde{\mu}$ inherits the decay
  property of \eqref{item:1} from $\mu$. By the Mass Distribution
  Principle it then follows that $\dim(G) = \dim(F_M) >
  \dim(G)$, which contradicts our original assumption. Therefore we conclude that $\dim(G)=1$.
\end{proof}

\section{Concluding remarks}
\label{sec:concluding-remarks}

We suspect that the rate of convergence in Theorem \ref{thm:main} can
be replaced by $(\log q)^{-1}$, at least in the case $\gamma = \delta
= 0$ and with the $\mathcal{D}_j$ growing at most geometrically. This is
certainly the case with \eqref{eq:6} for a single fixed $\alpha_i$ as
proved in \cite{MR1819996}. However it doesn't seem possible to prove this
using the method we have given here, without some significant modification.

We suspect that it is possible to improve the discrepancy estimate in
Corollary \ref{cor2} to replace the exponent $5/2 + \nu$ by $3/2 +\nu$
for $\mu_M$-almost every $x$. This is a natural conjecture in view of
the Lebesgue-almost sure discrepancy estimates of R. C. Baker
\cite{MR623668}. However, although it is an interesting problem in its
own right, for the Littlewood-type statements this would only give a
marginal improvement.

\def\cprime{$'$}
\providecommand{\bysame}{\leavevmode\hbox to3em{\hrulefill}\thinspace}
\providecommand{\MR}{\relax\ifhmode\unskip\space\fi MR }
\providecommand{\MRhref}[2]{%
  \href{http://www.ams.org/mathscinet-getitem?mr=#1}{#2}
}
\providecommand{\href}[2]{#2}

\end{document}